\documentclass{amsart}

\usepackage{amssymb}
\usepackage[hidelinks]{hyperref}
\usepackage{enumitem}
\usepackage{comment}
\usepackage{nicefrac}
\usepackage{bm}
\usepackage{mathrsfs}
\usepackage{graphicx}
\usepackage{cancel}
\usepackage{mathtools}
\usepackage[hidelinks]{hyperref}
\usepackage[textsize=small]{todonotes}
\usepackage{xcolor}

\theoremstyle{plain}
\newtheorem{thm}{Theorem}[section]

\newtheorem{corollary}[thm]{Corollary}
\newtheorem{prop}[thm]{Proposition}

\theoremstyle{definition}
\newtheorem{defi}[thm]{Definition}

\newtheorem{example}[thm]{Example}

\newtheorem{remark}[thm]{Remark}

\theoremstyle{remark}

\newtheorem{claim}{Claim}

\newcommand{\cF}{\mathcal{F}}
\newcommand{\F}{\cF}

\newcommand{\cI}{\mathcal{I}}
\newcommand{\I}{\cI}

\newcommand{\FIN}{\mathsf{Fin}} 
\newcommand{\Fin}{\FIN}

\DeclareMathOperator{\FS}{\mathsf{FS}}
\DeclareMathOperator{\IP}{\mathsf{IP}}
\DeclareMathOperator{\rr}{\mathsf{r}}

\begin{document}


\title[Sets of Ramsey-limit points and IP-limit points]{Sets of Ramsey-limit points and IP-limit points}


\author[R.~Filip\'{o}w]{Rafa\l{} Filip\'{o}w}
\address[Rafa\l{}~Filip\'{o}w]{Institute of Mathematics\\ Faculty of Mathematics, Physics and Informatics\\ University of Gda\'{n}sk\\ ul.~Wita Stwosza 57\\ 80-308 Gda\'{n}sk\\ Poland}
\email{Rafal.Filipow@ug.edu.pl}
\urladdr{\url{http://mat.ug.edu.pl/~rfilipow}}

\author[A.~Kwela]{Adam Kwela}
\address[Adam Kwela]{Institute of Mathematics\\ Faculty of Mathematics\\ Physics and Informatics\\ University of Gda\'{n}sk\\ ul.~Wita  Stwosza 57\\ 80-308 Gda\'{n}sk\\ Poland}
\email{Adam.Kwela@ug.edu.pl}
\urladdr{\url{https://mat.ug.edu.pl/~akwela}}

\author[P.~Leonetti]{Paolo Leonetti}
\address[Paolo Leonetti]{Department of Economics\\ Universit\'{a} degli Studi dell’Insubria\\ via Monte Generoso 71 \\ Varese 21100\\ Italy}
\email{leonetti.paolo@gmail.com}
\urladdr{\url{https://sites.google.com/site/leonettipaolo}}

\thanks{The second-listed author was supported by the Polish National Science Centre project OPUS No. 2024/53/B/ST1/02494.}


\date{\today}


\subjclass[2020]{Primary: 03E75, 05D10, 40A05, 54A20. 
Secondary: 03E05, 11B05, 28A05, 40A35.}





\keywords{
Ramsey convergence; 
IP-convergence; 
analytic sets; 
ideal limit points; 
Hindman ideal; 
partition regular functions. 
}


\begin{abstract}
Let $X$ be an uncountable Polish space and let $\mathcal{H}$ be the Hindman ideal, that is, the family of all $S\subseteq \omega$ which are not $\IP$-sets. For each sequence $x=(x_n)_{n \in \omega}$ taking values in $X$, let $\Lambda_{x}(\FS)$ be the set of $\IP$-limit points of $x$. Also, let $\Lambda_{x}(\mathcal{H})$ be the set of $\mathcal{H}$-limit points of $x$, that is, the set of ordinary limits of subsequences $(x_n)_{n \in S}$ with $S\notin \mathcal{H}$. After proving that these two notions do not coincide in general, we show that both families of nonempty sets of the type $\Lambda_{x}(\FS)$ and of the type $\Lambda_{x}(\mathcal{H})$ are precisely the class of nonempty analytic subsets of $X$. 

An analogous result holds also for Ramsey convergence. In the proofs, we use the concept of partition regular functions introduced in 
\cite{FKK23}, 
which provide a unified approach to these types of convergence. 
\end{abstract}


\maketitle




\section{Introduction}\label{sec:intro}

For each infinite subset $D$ of the nonnegative integers $\omega$, denote by $\FS(D)$ the family of all finite nonempty sums of distinct elements of $D$.  
A set $S\subseteq \omega$ is said to be an \emph{$\IP$-set} if it contains $\FS(D)$ for some infinite $D\subseteq\omega$. 
It is known that $\IP$-sets are examples of so-called Poincar\'{e} sequences which play an important role in the study of recurrences in topological dynamics, see e.g. \cite[p. 74]{MR0603625}. 
It follows from Hindman's theorem \cite{MR0349574} that the family
\begin{equation}\label{eq:defiH}
\mathcal{H}=\left\{S\subseteq \omega: S\text{ is not an $\IP$-set}\right\}
\end{equation}
is an ideal on $\omega$, that is, a family of subsets of $\omega$ closed under taking finite unions and subsets. Hereafter, $\mathcal{H}$ will be referred to as the \emph{Hindman ideal}, see e.g. \cite[p. 109]{MR2471564}. 

Given a sequence $x=(x_n)_{n \in \omega}$ taking values in a topological space $X$, we say that $\eta \in X$ is an \emph{$\IP$-limit point} (or an $\FS$\emph{-limit point}) 
of $x$ if there exists an infinite $D\subseteq \omega$ such that for every neighborhood $U$ of $\eta$ there exists $k \in \omega$ such that $x_n \in U$ for all $n \in \FS(D\setminus k)$; equivalently, the subsequence $(x_n)_{n \in \FS(D)}$ is \textquotedblleft $\IP$-convergent\textquotedblright\, to $\eta$, see e.g. \cite{MR0531271}. The set of such points is denoted by 
\begin{equation*}
    \Lambda_{x}(\FS)=\left\{\eta \in X: \eta \text{ is an }\mathrm{IP}\text{-limit point of }x\right\}.
\end{equation*}


In another direction, for each infinite set $D\subseteq \omega$, 
let $\mathsf{r}(D)$ be the family of subsets of $D$ with cardinality $2$. It follows by Ramsey's theorem that the family  
\begin{equation}\label{eq:defiR}
\mathcal{R}=\left\{S\subseteq \mathsf{r}(\omega): 
S\text{ does not contain }\mathsf{r}(D) 
\text{ for any infinite }D\subseteq \omega\right\}
\end{equation}
is an ideal on the countably infinite set $\mathsf{r}(\omega)$.  Hereafter, $\mathcal{R}$ is called the \emph{Ramsey ideal}, see e.g. \cite{MR3019575, alcantara-phd-thesis}. 
As before, if $y=(y_{\{n,m\}})_{\{n,m\} \in \mathsf{r}(\omega)}$ is a sequence taking values in a topological space $X$, we say that $\eta \in X$ is an $\mathsf{r}$\emph{-limit point} of $y$ if there exists an infinite set $D\subseteq \omega$ such that for every neighborhood $U$ of $\eta$ there exists $k \in \omega$ such that $y_{\{n,m\}} \in U$ for all $\{n,m\} \in \mathsf{r}(D\setminus k)$; equivalently, the subsequence $(y_{\{n,m\}})_{\{n,m\} \in \mathsf{r}(D)}$ is \textquotedblleft $\rr$-convergent\textquotedblright\, to $\eta$, see e.g. \cite{MR4741322, MR2948679}. The set of such points is denoted by 
\begin{equation*}
\Lambda_{y}(\mathsf{r})=
\left\{\eta \in X: \eta \text{ is an }\rr\text{-limit point of }y\right\}.
\end{equation*}

The main aim of this work is to show that 
each of the two families of nonempty sets of the type $\Lambda_{x}(\FS)$ and of the type $\Lambda_{y}(\rr)$ coincides with the class of nonempty analytic subsets of $X$, provided that $X$ is an uncountable Polish space (see Theorem \ref{thm:mainRamsey} and Theorem \ref{thm:mainHindman} below). 
This continues the line of research in \cite{MSP24, MR3883171, FKL24, MR4393937, MR4505549}
on the study of attainable sets of limit points with respect to ideal convergence. In fact, after recalling the necessary definitions in Section \ref{sec:preliminaries}, we show that variants of our main results hold also in the context of ideal limit points for the Hindman ideal $\mathcal{H}$ and Ramsey ideal $\mathcal{R}$ defined in \eqref{eq:defiH} and \eqref{eq:defiR}, respectively. In addition, in Section \ref{sec:Plike} we provide several characterizations of $P$-like properties, and specialize them to the cases of maps $\FS$ and $\rr$ (see Corollary \ref{cor:ojshdgkjshg} below). 

Following \cite{FKK23}, we use partition regular functions and study some of their properties. This gives us a unified approach to consider our questions for the types of convergence considered here. For additional motivations, we refer the reader to \cite{FKK23}.


\section{Preliminaries}\label{sec:preliminaries}


\subsection{Notations} 

The set of all nonnegative integers is denoted by $\omega$, and each nonnegative integer $n \in \omega$ is identified with the set $\{0,1,\ldots,n-1\}$. 
Given a set $X$, we denote by 
$X^\omega$ and $X^{<\omega}$
 the families of all infinite sequences $x:\omega\to X$ 
 and the families of all finite sequences $x:n\to X$ with $n\in \omega$, respectively. 
 For a finite sequence $s\in\omega^{<\omega}$ we denote its length by $|s|$. The infinite zero sequence $(0,0,\ldots)\in\omega^\omega$ is denoted by $0^\infty$.
Given a set $X$, we denote by $[X]^\kappa$ and $[X]^{<\kappa}$ the families of all subsets of $X$ of cardinality $\kappa$ and of cardinality strictly less than $\kappa$, respectively. 
In particular, we write $\Fin(X)=[X]^{<\omega}$ to denote the family of all finite subsets of $X$.

If $X$ is a metric space with metric $d$, we denote by $B(p,r)=\{x\in X:\ d(x,p)<r\}$ and $\overline{B}(p,r)=\{x\in X:\ d(x,p)\leq r\}$ the open ball and closed ball of radius $r>0$  and centered at $p \in X$, respectively. 
Lastly, if $X$ is a Polish space (that is, a separable and completely metrizable space), we use the standard notation $\Sigma^1_1(X)$ and $\Pi^1_1(X)$ (or simply $\Sigma^1_1$ and $\Pi^1_1$ if $X$ is understood) for the families of analytic and coanalytic subsets of $X$, respectively, see e.g.~\cite{MR1321597}.


\subsection{Partition regular functions} Let us start by recalling the main definition of partition regular function from \cite[Definition 3.1]{FKK23}.
\begin{defi}\label{defi:partitionregular}
    Let $\Omega,\Psi$ be two countably infinite sets and pick a nonempty family $\mathcal{F}\subseteq [\Omega]^\omega$ such that $A\setminus K \in \mathcal{F}$ for all $A \in \mathcal{F}$ and $K \in \Fin(\Omega)$. 
    We say that a function 
    \begin{equation*}
\rho: \mathcal{F}\to [\Psi]^\omega
\end{equation*}
is \emph{partition regular} if it satisfies the following properties\textup{:}
    \begin{enumerate}[label={\rm (\roman{*})}]
\item [$\mathrm{(M)}$] $\rho(E)\subseteq \rho(F)$ for all $E,F \in \mathcal{F}$ such that $E\subseteq F$;
\item [$\mathrm{(R)}$] for every $F \in \mathcal{F}$ and $A,B \subseteq \Psi$ such that $\rho(F)=A\cup B$, there exists $E \in \mathcal{F}$ such that either $\rho(E)\subseteq A$ or $\rho(E)\subseteq B$;
\item [$\mathrm{(S)}$] for every $F \in \mathcal{F}$ there exists a subset $E\subseteq F$, $E\in \mathcal{F}$ such that for every $a \in \rho(E)$ there exists $K \in \Fin(\Omega)$ for which $a\notin \rho(E\setminus K)$. 
\end{enumerate}
\end{defi}

For each partition regular function $\rho$ as in Definition \ref{defi:partitionregular}, we associate the family
\begin{equation}\label{eq:generatedideal}
\mathcal{I}_\rho=\left\{S\subseteq \Psi: 
\forall_{F \in \mathcal{F}}\,\,\rho(F)\not\subseteq S
\right\}.
\end{equation}
It is not difficult to see that $\mathcal{I}_\rho$ is an ideal on $\Psi$, that is, a proper subset of $\mathcal{P}(\Psi)$ which contains $\Fin(\Psi)$ and is closed under taking  finite unions and subsets; 
vice versa, for each ideal $\mathcal{I}$ on $\Psi$, we have $\mathcal{I}=\mathcal{I}_{\rho_\I}$ for the partition regular function $\rho_{\mathcal{I}}: \mathcal{I}^+\to [\Psi]^\omega$ defined by 
$\rho_{\mathcal{I}}(F)=F$
where $\mathcal{I}^+=\mathcal{P}(\Psi)\setminus \mathcal{I}$; see \cite[Proposition 3.3]{FKK23} for details. 

As anticipated in Section \ref{sec:intro}, in this work we are mainly interested in two partition regular functions $\rr:[\omega]^\omega\to[[\omega]^2]^\omega$ and $\FS:[\omega]^\omega\to[\omega]^\omega$ defined by 
\begin{equation*}
\rr(D)=[D]^2
\quad \text{ and }\quad 
\FS(D)=\left\{ \sum\nolimits_{n\in\alpha}n:\alpha\in[D]^{<\omega}\setminus\{\emptyset\}\right\}
\end{equation*}
for all $D \in [\omega]^\omega$, respectively, see \cite[Theorem 3.4 and Theorem 3.6]{FKK23}. 
Following the notation in \eqref{eq:generatedideal}, the associated ideals
\begin{equation*}
\mathcal{R}=\I_{\rr}
\quad \text{ and }\quad 
\mathcal{H}=\I_{\FS}
\end{equation*}
are called \emph{Ramsey ideal} and \emph{Hindman ideal}, respectively; see \cite{MR2471564, MR3019575, alcantara-phd-thesis}. 


\subsection{Families of cluster and limit points} We now present our main definitions of $\rho$-cluster points and $\rho$-limit points. Hereafter, given a sequence $x: \Psi\to X$, we usually write $x_s=x(s)$ for each $s \in \Psi$. 

\begin{defi}\label{defi:rhoclusterpoint}
Let $x: \Psi\to X$ be a sequence in a topological space $X$ 
and pick a partition regular function $\rho: \mathcal{F}\to [\Psi]^\omega$. 
Denote by $\Gamma_x(\rho)$ the set of $\bm{\rho}$\textbf{-cluster points} of $x$, that is, the set of all $\eta \in X$ such that, for all neighborhoods $U$ of $\eta$, 
\begin{equation*}
    \exists_{F \in \mathcal{F}} \quad 
    \rho(F)\subseteq \left\{s \in \Psi: x_s \in U\right\}. 
\end{equation*}

    We also write $\mathscr{C}_X(\rho)$ for the family of sets of $\rho$-cluster points of sequences with values in $X$, together with the empty set, that is, 
\begin{equation*}
    \mathscr{C}_X(\rho)=\left\{A\subseteq X: \exists_{x \in X^\Psi}\,\,A=\Gamma_x(\rho)\right\}\cup \{\emptyset\}. 
\end{equation*}
    If the topological space $X$ is understood, we simply write $\mathscr{C}(\rho)$.
\end{defi}
Equivalently, $\eta$ is a $\rho$-cluster point of $x$ if $\left\{s \in \Psi: x_s \in U\right\}\notin \mathcal{I}_\rho$ for all neighborhoods $U$ of $\eta$, where $\I_\rho$ is the ideal generated by $\rho$ as in \eqref{eq:generatedideal}. In particular, the definition of $\rho$-cluster point depends only on the ideal generated by $\rho$, hence it makes sense to write $\Gamma_x(\I_\rho)$. In fact, if $\I$ is an ideal on $\Psi$, elements of $\Gamma_x(\I)=\Gamma_x(\rho_\I)$ are usually called $\I$-cluster points, see e.g. \cite{MR3883171, MR3920799}. It is worth noting that the case of an empty set of $\rho$-cluster points is not really interesting: indeed, it is known that, if $X$ is a metric space and $\I$ is an ideal on $\omega$, then there exists a sequence $x: \omega\to X$ such that $\Gamma_x(\I)=\emptyset$ if and only if $X$ is not compact, see \cite[Proposition 5.1]{FKL24}. 


For the next definition, let $X$ be a topological space, $\rho: \mathcal{F}\to [\Psi]^\omega$ be a partition regular function, and pick a sequence $x: \Psi\to X$. A subsequence $y$ of $x$ is said to be a $\rho$\emph{-subsequence} of $x$ if there exists $F \in \mathcal{F}$ such that $y=(x_s)_{s \in \rho(F)}$. Informally, a $\rho$-subsequence is a subsequence of $x$ with an index set which is \textquotedblleft not small\textquotedblright\, with respect to $\rho$. Accordingly, a $\rho$-subsequence $y=x\upharpoonright \rho(F)$ is said to be $\rho$\emph{-convergent to} $\eta \in X$, shortened as $\rho\text{-}\lim y=\eta$, if for every neighborhood $U$ of $\eta$ there exists $K \in \Fin(\Omega)$ such that 
$y_s \in U$ for every $s \in \rho(F\setminus K)$. 
Notice that this convergence depends on the choice of $F$. In fact, it may happen that $\rho(F)=\rho(G)$ for some $F,G \in \mathcal{F}$, yet $(y_s)_{s \in \rho(F)}$ is $\rho$-convergent while $(y_s)_{s \in \rho(G)}$ is not, see \cite{FKK23}. 

\begin{defi}\label{defi:rholimitpoint}
Let $x: \Psi\to X$ be a sequence in a topological space $X$ 
and pick a partition regular function $\rho: \mathcal{F}\to [\Psi]^\omega$. 
Denote by 
$\Lambda_x(\rho)$ the set of $\bm{\rho}$\textbf{-limit points} of $x$, that is, the set of all $\eta \in X$ such that 
\begin{equation*}
    \rho\text{-}\lim y=\eta
\end{equation*}
for some $\rho$-subsequence $y$ of $x$. We also write $\mathscr{L}_X(\rho)$ for the family of sets of $\rho$-limit points of sequences with values in $X$ together with the empty set, that is, 
    \begin{equation}\label{eq:limitpointsfamily}
    \mathscr{L}_X(\rho)=\left\{A\subseteq X: \exists_{x \in X^\Psi}\, A=\Lambda_x(\rho)\right\}\cup \{\emptyset\}. 
    \end{equation}
    If the topological space $X$ is understood, we simply write $\mathscr{L}(\rho)$. 

In the case where $\rho=\rho_{\I}$ for some ideal $\I$ on $\omega$, we define 
\begin{equation*}
\Lambda_x(\I)=\Lambda_x(\rho_{\I})
\quad \text{ and }\quad 
\mathscr{L}_X(\I)=\mathscr{L}_X(\rho_{\I}).
\end{equation*}
Elements of $\Lambda_x(\I)$ are called $\I$\emph{-limit points} of $x$. 
\end{defi}

More explicitly, $\eta$ is a $\rho$-limit point of $x$ if there exists $F \in \mathcal{F}$ such that, for each neighborhood $U$ of $\eta$, there is a finite set $K \in \Fin(\Omega)$ such that $x_s \in U$ for all $s \in \rho(F\setminus K)$. In the case where $\rho=\rho_\I$ for some ideal $\I$, it is easy to see that $\eta$ is an $\I$-limit point 
if and only if there exists $E\in \I^+$ such that the subsequence $y=x\upharpoonright E$ is convergent to $\eta$; hence, this coincides with the classical notion of $\I$-limit point studied e.g. in \cite{MSP24, MR3883171, FKL24, MR3883170}. 
Unlike the case of $\rho$-cluster points, realizing the empty set 
in \eqref{eq:limitpointsfamily} 
as a set of the form $\Lambda_x(\rho)$ is much more difficult to handle. 
Indeed, if $\I$ is an ideal on $\omega$, several implications on the existence of a sequence $x$ such that $\Lambda_x(\I)=\emptyset$ can be found in \cite[Proposition 5.3 and Figure 3]{FKL24}. Recently, significant progress on this topic was made in \cite{Gosia}. 
It is worth noting that families of $\I$-limit points and $\I$-cluster points for an arbitrary ideal $\I$ were introduced  in \cite{MR1844385}, however they were already considered in \cite{MR1181163} in the case of the ideal of all sets of asymptotic density zero.

\begin{prop}
\label{prop-lambda-ideal-in-lambda-rho}
Let $x: \Psi\to X$ be a sequence in a topological space $X$ 
and pick a partition regular function $\rho: \mathcal{F}\to [\Psi]^\omega$. Then 
\begin{equation*}
\Lambda_x({\I_\rho})\subseteq\Lambda_x(\rho)\subseteq\Gamma_x(\rho). 
\end{equation*}
\end{prop}

\begin{proof}
First, suppose that $\eta \in \Lambda_x({\I_\rho})$, i.e., there exists $S\in\I_\rho^+$ such that $\lim x\upharpoonright S=\eta$. Then there is $F\in\cF$ such that $\rho(F)\subseteq S$, which gives $\lim x\upharpoonright \rho(F)=\eta$. Since $\rho$ is partition regular, there is $E\in\cF$ such that $E\subseteq F$ (which implies $\rho(E)\subseteq\rho(F)$) and for all $a\in\rho(E)$ there is a finite $K\subseteq \Omega$ with $a\notin\rho(E\setminus K)$. Then it is easy to check that the $\rho$-subsequence $x\upharpoonright \rho(E)$ is $\rho$-convergent to $\eta$. Hence, $\eta\in\Lambda_x(\rho)$. This implies that $\Lambda_x({\I_\rho})\subseteq\Lambda_x(\rho)$. 

Second, suppose that $\eta\in\Lambda_x({\rho})$, i.e., there exists $F\in\cF$ such that the $\rho$-subsequence $x\upharpoonright \rho(F)$ is $\rho$-convergent to $\eta$. Then for each neighborhood $U$ of $\eta$ we can find a finite $K\subseteq \Omega$ such that $x_s\in U$ for all $s\in\rho(F\setminus K)$. Since $\rho(F\setminus K)\in\I^+_{\rho}$, we conclude that $\eta\in\Gamma_x(\rho)$. Therefore $\Lambda_x(\rho)\subseteq\Gamma_x(\rho)$. 
\end{proof}

The next example shows that $\Lambda_x(\I_\rho)$ does not coincide, in general, with $\Lambda_x(\rho)$. 
\begin{example}
For each nonzero $n \in \omega$, let $\nu_2(n)$ be the $2$-adic valuation of $n$, that is, the maximal $k \in \omega$ such that $2^k$ divides $n$, and define $S_m=\{n \in \omega\setminus \{0\}: \nu_2(n)=m\}$ for each $m \in \omega$. Set $\rho=\FS$ and $X=[0,1]$, so that $\mathcal{F}=[\omega]^\omega$ and $\mathcal{H}=\I_\rho$. We claim that there exists a sequence $x\in X^\omega$ which has no $\mathcal{I}_\rho$-limit points and, on the other hand, it admits a $\rho$-subsequence which is $\rho$-convergent. In particular, this would imply that
\begin{equation*}
\Lambda_x({\mathcal{H}})=\emptyset
\quad \text{ and }\quad 
\Lambda_x({\FS})\neq \emptyset.
\end{equation*}
We are going to show that the sequence $x=(x_n)_{n \in \omega}$ defined by $x_0=1/3$ and $x_n=2^{-\nu_2(n)}$ for each nonzero $n \in \omega$ satisfies our claim. 

To this end, we show first that $x$ has no $\mathcal{H}$-limit points. Indeed, suppose for the sake of contradiction that there exists an infinite $E\subseteq \omega$ such that $x\upharpoonright \FS(E)$ is convergent (in the ordinary sense). Since the image of $x$ equals $\{1/3\}\cup \{2^{-n}:n\in\omega\}$, there are only the following two possibilities.
\begin{enumerate}[label={\rm (\roman{*})}]
    \item The sequence $(x_n)_{n \in \FS(E)}$ is eventually constant. In other words, there exist $k,m \in \omega$ such that $\FS(E)\setminus k\subseteq S_m$. In particular, since $E\subseteq \FS(E)$, we have $E\setminus k \subseteq S_m$. Now, pick distinct $a,b \in E\setminus k$, so that $\nu_2(a)=\nu_2(b)=m$. Then $a+b \in  \FS(E)\setminus k\subseteq S_m$ and, on the other hand, $\nu_2(a+b)\ge m+1$, 
    contradicting $a+b \in S_m$. 
    
    \item The sequence $(x_n)_{n \in \FS(E)}$ is convergent to zero. Take any $a\in E\setminus\{0\}$ and pick $k\in \omega$ such that $a\in S_k$. Since $(x_n)_{n \in \FS(E)}$ converges to zero, $\FS(E)\cap S_m$ has to be finite for each $m\in\omega$. In particular, since $E\subseteq\FS(E)$, we have that $E\cap S_m$ is finite for each $m\in\omega$. Then $E\setminus\bigcup_{m\leq k}S_m$ is infinite and for each $b\in E\setminus\bigcup_{m\leq k}S_m$ we have $a+b\in\FS(E)\cap S_k$, which contradicts the fact that $\FS(E)\cap S_k$ is finite. 
\end{enumerate}

Lastly, we show that, if $F=\{2^n: n \in \omega\}$ (so that $\FS(F)=\omega\setminus \{0\}$), then the $\FS$-subsequence $y=x\upharpoonright \FS(F)$ is $\FS$-convergent to $0$. Indeed, for each $k \in \omega$ and each $n \in \FS(F\setminus 2^k)$, we have $\nu_2(n)\ge k$, hence $y_n \le 2^{-k}$. Therefore $\FS\text{-}\lim y=0$. 
\end{example}

The main goal of this work is to show that, for each nonempty analytic subset $C$ of an uncountable Polish space $X$, there exist sequences $x$ and $y$ 
taking values in $X$ which satisfy 
\begin{equation*}
C=\Lambda_{x}(\FS)=\Lambda_x({\mathcal{H}}) 
\quad \text{ and }\quad  
C=\Lambda_{y}(\rr)=\Lambda_y({\mathcal{R}}).
\end{equation*}
As an application, we obtain 
\begin{equation*}
\mathscr{L}(\FS)=\mathscr{L}(\mathcal{H})=
\mathscr{L}(\rr)=\mathscr{L}(\mathcal{R})=
\Sigma^1_1,
\end{equation*}
see Theorem \ref{thm:mainRamsey} and Theorem \ref{thm:mainHindman} below. 


\section{P-like properties}\label{sec:Plike}

It will be useful to recall the following definitions, see
\cite[Section 6]{FKK23}. 
For a partition regular function
$\rho:\cF\to[\Psi]^\omega$, we write 
$\rho(F)\subseteq^\rho B$
if there is a finite set $K\subseteq \Omega$ with $\rho(F\setminus K)\subseteq B$.
We remark that 
``$\rho(F)\subseteq^\rho B$'' 
is a relation between $F$ and $B$, and not between $\rho(F)$ and $B$ (in fact, it is possible that $\rho(F)=\rho(G)$ and $\rho(F)\subseteq^\rho B$ but $\rho(G)\not\subseteq^\rho B$, see \cite[Remark after Definition 6.3]{FKK23}).   

A partition regular function $\rho: \mathcal{F}\to [\Psi]^\omega$ is
\begin{enumerate}[label={\rm (\roman{*})}]
\item 
\emph{$P^+$} if for every $\subseteq$-decreasing sequence $(A_n)_{n\in\omega}\in (\I_{\rho}^+)^\omega$  
there exists $F\in\cF$ such that $\rho(F)\subseteq^\rho A_n$
for each $n\in\omega$;
\item 
\emph{$P^-$} if for every $\subseteq$-decreasing sequence $(A_n)_{n\in\omega}\in (\I_{\rho}^+)^\omega$ with $A_n\setminus A_{n+1}\in \I_{\rho}$ for each $n\in\omega$ 
there exists $F\in\cF$ such that $\rho(F)\subseteq^\rho A_n$
for each $n\in\omega$;
\item 
\emph{$P^{\,|}$} if for every $\subseteq$-decreasing sequence $(A_n)_{n\in\omega}\in (\I_{\rho}^+)^\omega$ with $A_n\setminus A_{n+1}\in \I_{\rho}^+$ for each $n\in\omega$ 
there exists $F\in\cF$ such that $\rho(F)\subseteq^\rho A_n$
for each $n\in\omega$.
\end{enumerate}

\begin{prop}
\label{P-like_basic}
Let $\rho: \mathcal{F}\to [\Psi]^\omega$ be a partition regular function. Then $\rho$ is $P^+$ if and only if it is both $P^{\,|}$ and $P^-$.
\end{prop}

\begin{proof}
\textsc{Only If part.} This is obvious.

\textsc{If part.} 
Suppose that $\rho$ is $P^{\,|}$ and $P^-$. Fix a $\subseteq$-decreasing sequence $(A_n)_{n\in\omega}\in (\I_{\rho}^+)^\omega$. Define 
\begin{equation*}
T=\{0\}\cup \{n\in\omega: A_n\setminus A_{n+1}\in \I_{\rho}^+\}.
\end{equation*}

If $T$ is finite, define $k=\max T$ and $B_n=A_{k+1+n}$ for all $n\in\omega$. Then $(B_n)_{n\in\omega}$ is a $\subseteq$-decreasing sequence with $B_n\setminus B_{n+1}\in \I_{\rho}$ for each $n\in\omega$. Since $\rho$ is $P^-$,
there exists $F\in\cF$ such that $\rho(F)\subseteq^\rho B_n$
for each $n\in\omega$, hence $\rho(F)\subseteq^\rho A_n$ for each $n\in\omega$.

If $T$ is infinite, let $(t_n)_{n\in\omega}$ be the increasing enumeration of $T$ and define $C_n=A_{t_n}$ for all $n\in\omega$. Then $(C_n)_{n\in\omega}$ is a $\subseteq$-decreasing sequence with $C_n\setminus C_{n+1}\supseteq A_{t_n}\setminus A_{t_n+1}\in \I^+_{\rho}$ for each $n\in\omega$. Applying the fact that $\rho$ is $P^{\,|}$, there exists $E\in\cF$ such that $\rho(E)\subseteq^\rho C_n$
for each $n\in\omega$. Also in this case $\rho(E)\subseteq^\rho A_n$ for each $n\in\omega$. Therefore $\rho$ is $P^+$.
\end{proof}

\begin{prop}
\label{P-like_for_spaces}
Let $X$ be a metric space with an accumulation point, and let  $\rho: \mathcal{F}\to [\Psi]^\omega$ be a partition regular function. Then the following hold.
\begin{enumerate}[label={\rm (\roman{*})}]
    \item $\rho$ is $P^+$ if and only if $\Lambda_x(\rho)=\Gamma_x(\rho)$ for every sequence $x\in X^\Psi$.\label{P-like_for_spaces:+}
    
    \item $\rho$ is $P^{\,|}$ if and only if $\Lambda_x(\rho)$ is closed for every sequence $x\in X^\Psi$.\label{P-like_for_spaces:|}
    
    \item Suppose that $X$ is locally compact. Then $\rho$ is $P^-$ if and only if each isolated point in $\Gamma_x(\rho)$ belongs to $\Lambda_x(\rho)$ for every sequence $x\in X^\Psi$.\label{P-like_for_spaces:-}
\end{enumerate}
\end{prop}

\begin{proof}
In the following proofs, let $d$ be the metric on $X$. 

\ref{P-like_for_spaces:+} 
\textsc{Only If part.} 
Fix a sequence $x=(x_s)_{s\in\Psi}$ in $X$. By Proposition \ref{prop-lambda-ideal-in-lambda-rho} we have $\Lambda_x(\rho)\subseteq\Gamma_x(\rho)$, so we only need to show that $\Gamma_x(\rho)\subseteq \Lambda_x(\rho)$. If $\Gamma_x(\rho)=\emptyset$, the claim is clear. Otherwise, fix $p\in \Gamma_x(\rho)$. Then for every $n\in\omega$ we have $A_n=\{s\in \Psi:\ d(x_s,p)<2^{-n}\}\in\I_\rho^+$. Moreover, $(A_n)_{n\in\omega}$ is $\subseteq$-decreasing. Since $\rho$ is $P^+$, there is $F\in\cF$ such that $\rho(F)\subseteq^\rho A_n$ for all $n\in\omega$. Now, it is enough to show that $(x_s)_{s\in\rho(F)}$ is $\rho$-convergent to $p$. To this aim, let $U$ be a neighborhood of $p$ and pick $n\in\omega$ such that $B(p,2^{-n})\subseteq U$. Then there is a finite set $K\subseteq\Omega$ such that $\rho(F\setminus K)\subseteq A_n$. It follows that $x_s\in B(p,2^{-n})\subseteq U$ for all $s\in \rho(F\setminus K)$, hence $(x_s)_{s\in\rho(F)}$ is $\rho$-convergent to $p$.

\textsc{If part.} Let us suppose that $\rho$ is not $P^+$, and let $(A_n)_{n\in\omega}\in (\I_{\rho}^+)^\omega$ be a $\subseteq$-decreasing sequence such that for every $F\in\cF$ we have $\rho(F)\not\subseteq^\rho A_n$ for some $n\in\omega$. Let also $p$ be an accumulation point of $X$ and pick an injective sequence $(y_n)_{n\in\omega}$ such that the sequence $(\delta_n)_{n\in\omega}$ given by $\delta_n=d(y_n,p)$ is decreasing and tends to zero. Define the sequence $x=(x_s)_{s\in\Psi}$ by 
\begin{equation*}
x_{s}=\begin{cases}
    \,y_0 &\text{if }s\in\Psi\setminus A_0,\\
    \,p &\text{if }s\in\bigcap_{n\in\omega}A_n,\\
    \,y_n &\text{if }s\in A_n\setminus A_{n+1}.
\end{cases}
\end{equation*}
We will show that $p\in\Gamma_x(\rho)$ and that $p\notin\Lambda_x(\rho)$. 
 
In fact, let $U$ be a neighborhood of $p$ and find $n\in\omega$ such that the closed ball $\overline{B}(p,\delta_n)$ is contained in $U$. Then $x_s\in \overline{B}(p,\delta_n)\subseteq U$ for all $s\in A_n$. Since $A_n\in\I_\rho^+$, we obtain that $p\in\Gamma_x(\rho)$.

Lastly, to prove that $p\notin\Lambda_x(\rho)$, fix $F\in\cF$. We are going to show that the subsequence $(x_s)_{s\in\rho(F)}$ does not $\rho$-converge to $p$. By the choice of $(A_n)_{n\in\omega}$, there is $n\in\omega$ such that $\rho(F\setminus K)\not\subseteq A_n$ for all finite $K\subseteq\Omega$. Define $U=B(p,\delta_n)$. Then $U$ is a neighborhood of $p$. Fix a finite set $K\subseteq\Omega$. There is $s\in \rho(F\setminus K)\setminus A_n$. Then $s\in \rho(F\setminus K)$ and $x_s\in\{y_0,y_1,\ldots,y_{n-1}\}$, so $x_s\notin U$. By the arbitrariness of $K$, we conclude that $(x_s)_{s\in\rho(F)}$ does not $\rho$-converge to $p$.

\ref{P-like_for_spaces:|} 
\textsc{Only If part.} Fix a sequence $x=(x_s)_{s\in\Psi}$ in $X$ and let $(y_n)_{n\in\omega}$ be a sequence in $\Lambda_x(\rho)$ converging to some $p\in X$. Then for each $n\in\omega$ there is $F_n\in\cF$ such that the $\rho$-subsequence $(x_s)_{s\in\rho(F_n)}$ is $\rho$-convergent to $y_n$. We need to show that $p\in\Lambda_x(\rho)$. This is clear if $p=y_n$ for some $n\in\omega$, so we can assume that $p\neq y_n$ for all $n\in\omega$. Then it is possible to find a decreasing sequence $(\delta_n)_{n\in\omega}$ of positive reals such that $\lim_{n}\delta_n=0$ and for each $n\in\omega$ there is $k\in\omega$ such that $d(y_k,p)\in (\delta_{n+1},\delta_n)$. Define $A_n=\{s\in \Psi:\ d(x_s,p)<\delta_n\}$ for all $n\in\omega$. Then $(A_n)_{n\in\omega}$ is $\subseteq$-decreasing. Moreover, for every $n\in\omega$ we have $A_n\setminus A_{n+1}\in\I^+_\rho$. Indeed, if $n\in\omega$ then there is $k\in\omega$ such that $\delta_{n+1}<d(y_k,p)<\delta_n$, so $U=B(p,\delta_n)\setminus\overline{B}(p,\delta_{n+1})=\{x\in X:\ \delta_{n+1}<d(x,p)<\delta_n\}$ is a neighborhood of $y_k$ and we can find a finite $K\subseteq\Omega$ such that $x_s\in U$ for all $s\in\rho(F_k\setminus K)$. Thus, $\rho(F_k\setminus K)\subseteq A_n\setminus A_{n+1}$ and consequently $A_n\setminus A_{n+1}\in\I^+_\rho$ (in particular, $A_n\in\I^+_\rho$).

Since $\rho$ is $P^{\,|}$, there is $F\in\cF$ such that $\rho(F)\subseteq^\rho A_n$ for all $n\in\omega$. We will show that $(x_s)_{s\in\rho(F)}$ is $\rho$-convergent to $p$. Let $U$ be a neighborhood of $p$ and find $n\in\omega$ such that the ball $B(p,\delta_n)$ is contained in $U$. Then there is a finite set $K\subseteq\Omega$ such that $\rho(F\setminus K)\subseteq A_n$. It follows that $x_s\in B(p,\delta_n)\subseteq U$ for all $s\in \rho(F\setminus K)$, hence $(x_s)_{s\in\rho(F)}$ is $\rho$-convergent to $p$. Therefore $p \in \Lambda_x(\rho)$.

\textsc{If part.} Assume that $\rho$ is not $P^{\,|}$. Hence it is possible to fix a $\subseteq$-decreasing sequence of sets $(A_n)_{n\in\omega}\in (\I_{\rho}^+)^\omega$ such that
\begin{itemize}
    \item $A_n\setminus A_{n+1}\in\I^+_\rho$ for all $n\in\omega$,
    \item for every $F\in\cF$ there is $n\in\omega$ such that $\rho(F)\not\subseteq^\rho A_n$.
\end{itemize}  
Let also $p$ be an accumulation point of $X$ and pick an injective sequence $(y_n)_{n\in\omega}$ such that the sequence $(\delta_n)_{n\in\omega}$ defined by $\delta_n=d(y_n,p)$ is strictly decreasing and has limit zero. At this point, define the sequence $x=(x_s)_{s\in\Psi}$ by 
\begin{equation*}
x_{s}=\begin{cases}
    \,y_0 &\text{if }s\in\Psi\setminus A_0,\\
    \,p &\text{if }s\in\bigcap_{n\in\omega}A_n,\\
    \,y_n &\text{if }s\in A_n\setminus A_{n+1}.
\end{cases}
\end{equation*}
To complete the proof, it will be enough to show that 
\begin{equation}\label{eq:claimPropostion24ii}
\{y_n:\ n\in\omega\}\subseteq\Lambda_x(\rho)
\quad \text{ and }\quad 
p\notin\Lambda_x(\rho).
\end{equation}
To this aim, fix $n \in \omega$. 
Then, $A_n\setminus A_{n+1}\in\I^+_\rho$, hence there is $F\in\cF$ such that $\rho(F)\subseteq A_n\setminus A_{n+1}$. Since $x_s=y_n$ for all $s\in\rho(F)$, the $\rho$-subsequence $(x_s)_{s\in\rho(F)}$ is $\rho$-convergent to $y_n$. This proves the first part of \eqref{eq:claimPropostion24ii}. The second part of \eqref{eq:claimPropostion24ii} 
can be shown similarly as in the proof of the \textsc{If} part of item \ref{P-like_for_spaces:+}.

\ref{P-like_for_spaces:-} 
\textsc{Only If part.} 
Fix a sequence $x=(x_s)_{s\in\Psi}$ in $X$ and let $p$ be an isolated point of $\Gamma_x(\rho)$. Taking into account that $X$ is locally compact, it is possible to fix an integer $n_0\in\omega$ such that $\overline{B}(p,2^{-n_0})$ is compact and $\overline{B}(p,2^{-n_0})\cap\Gamma_x(\rho)=\{p\}$. Since $p\in\Gamma_x(\rho)$, 
for every $n\in\omega$ 
we have that 
\begin{equation*}
A_n=\{s\in\Psi:\ x_s\in \overline{B}(p,2^{-(n_0+n)})\}\in\I^+_\rho.
\end{equation*}
In addition, $(A_n)_{n\in\omega}$ is $\subseteq$-decreasing. We are going to show that $A_n\setminus A_{n+1}\in\I_\rho$ for all $n\in\omega$. Indeed, suppose for the sake of contradiction that there is $n\in\omega$ such that $A_n\setminus A_{n+1}\in \I_\rho^+$.
Then $Z=\overline{B}(p,2^{-(n_0+n)})\setminus B(p,2^{-(n_0+n+1)})$ is compact and $x_s\in Z$ for all $s\in A_n\setminus A_{n+1}\in\I^+$. Thanks to 
\cite[Lemma 3.1(vi)]{MR3920799}, 
it is possible to fix a point  $p'\in\Gamma_x(\rho)\cap Z$. Then $p'\neq p$ and $p'\in Z\subseteq \overline{B}(p,2^{-(n_0+n)})\subseteq \overline{B}(p,2^{-n_0})$. This contradicts the choice of $n_0$. 

Since $\rho$ is $P^-$, there is $F\in\cF$ such that $\rho(F)\subseteq^\rho A_n$ for all $n\in\omega$. Now, we are going to show that $(x_s)_{s\in\rho(F)}$ is $\rho$-convergent to $p$. To this aim, let $U$ be a neighborhood of $p$ and find $n\in\omega$ such that the closed ball $\overline{B}(p,2^{-(n_0+n)})$ is contained in $U$. There is a finite set $K\subseteq\Omega$ such that $\rho(F\setminus K)\subseteq A_n$. Then $x_s\in \overline{B}(p,2^{-(n_0+n)})\subseteq U$ for all $s\in \rho(F\setminus K)$, so $(x_s)_{s\in\rho(F)}$ is $\rho$-convergent to $p$. 

\textsc{If part.}
Assume that $\rho$ is not $P^-$ and let $(A_n)_{n\in\omega}\in (\I_{\rho}^+)^\omega$ be a $\subseteq$-decreasing sequence such that:
\begin{itemize}
    \item $A_n\setminus A_{n+1}\in\I_\rho$ for all $n\in\omega$,
    \item for every $F\in\cF$ there is $n\in\omega$ such that $\rho(F)\not\subseteq^\rho A_n$.
\end{itemize}  
Let also $p$ be an accumulation point of $X$ and pick an injective sequence $(y_n)_{n\in\omega}$ such that the sequence $(\delta_n)_{n\in\omega}$ given by $\delta_n=d(y_n,p)$ is decreasing and tends to zero. Define $x=(x_s)_{s\in\Psi}$ by 
\begin{equation*}
x_{s}=\begin{cases}
    \,y_0 &\text{if }s\in\Psi\setminus A_0,\\
    \,p &\text{if }s\in\bigcap_{n\in\omega}A_n,\\
    \,y_n &\text{if }s\in A_n\setminus A_{n+1}.
\end{cases}
\end{equation*}
Note that $\Gamma_x(\rho)\subseteq\{y_n:\ n\in\omega\}\cup\{p\}$. To finish the proof, we will show that $p\in\Gamma_x(\rho)$, $y_n\notin\Gamma_x(\rho)$ for all $n\ge 1$ (so that $p$ is an isolated point of $\Gamma_x(\rho)$) and that $p\notin\Lambda_x(\rho)$.

To this aim, let $U$ be a neighborhood of $p$ and find $n\in\omega$ such that the ball $\overline{B}(p,\delta_n)$ is contained in $U$. Then $x_s\in \overline{B}(p,\delta_n)\subseteq U$ for all $s\in A_n$. Since $A_n\in\I_\rho^+$, we obtain that $p\in\Gamma_x(\rho)$.

In addition, pick an integer $n\ge 1$ and 
find a neighborhood $U$ of $y_n$ such that $p\notin U$ and $y_k\notin U$ for every $k\neq n$. Then $\{s\in\Psi:\ x_s\in U\}=A_n\setminus A_{n+1}\in\I_\rho$. Hence $y_n\notin\Gamma_x(\rho)$. 

Lastly, the proof that $p\notin\Lambda_x(\rho)$ goes similarly as in the \textsc{If} part of item \ref{P-like_for_spaces:+}.
\end{proof}

\begin{remark}
The assumption about local compactness in Proposition \ref{P-like_for_spaces}\ref{P-like_for_spaces:-} cannot be dropped. Indeed, in \cite[Example 3.8]{FKL24} it is shown that there are an ideal $\I$, a space $X$, and a sequence $x\in X^\omega$ such that
\begin{itemize}
    \item $X$ is metric with an accumulation point, but it is not locally compact,
    \item $\I$ is $P^-$, which is equivalent to $\rho_\I$ being $P^-$ (by \cite[Proposition 6.5(2)]{FKK23}),
    \item there is an isolated point of $\Gamma_x(\I)$ which does not belong to $\Lambda_x(\I)$.
\end{itemize}
Since $\Gamma_x(\rho_\I)=\Gamma_x(\I_{\rho_\I})=\Gamma_x(\I)$ and $\Lambda_x(\I)=\Lambda_x(\rho_\I)$, we get an example of a metric space $X$ with an accumulation point, a sequence $x$ in $X$ and a $P^-$ partition regular function $\rho_\I$ such that there is an isolated point of $\Gamma_x(\I_{\rho_\I})$ which does not belong to $\Lambda_x(\rho_\I)$.
\end{remark}

\begin{corollary}\label{cor:ojshdgkjshg}
Let $X$ be a metric space with an accumulation point. Then the following hold. 
\begin{enumerate}[label={\rm (\roman{*})}]
\item $\Lambda_x(\rr)\neq\Gamma_x(\rr)$  for some sequence $x\in X^{[\omega]^2}$\textup{.}
\item $\Lambda_x(\rr)$ is not closed for some sequence $x\in X^{[\omega]^2}$\textup{.}
\item If $X$ is locally compact, then each isolated point in $\Gamma_x(\rr)$ belongs to $\Lambda_x(\rr)$ for every $x\in X^{[\omega]^2}$.
\end{enumerate}
Similarly, the following hold.
\begin{enumerate}[label={\rm (\roman{*}$^\prime$)}]
\item $\Lambda_x(\FS)\neq\Gamma_x(\FS)$ for some sequence $x\in X^{\omega}$\textup{.}
\item $\Lambda_x(\FS)$ is not closed for some sequence $x\in X^{\omega}$\textup{.}
\item If $X$ is locally compact, then each isolated point in $\Gamma_x(\FS)$ belongs to $\Lambda_x(\FS)$ for every $x\in X^{\omega}$.
\end{enumerate}
\end{corollary}

\begin{proof}
Thanks to \cite[Proposition 6.7]{FKK23}, both partition regular functions $\rr$ and $\FS$ are $P^-$, but not $P^+$. Then $\rr$ and $\FS$ are not $P^{\,|}$ by Proposition \ref{P-like_basic}. To conclude, all claims follow by Proposition \ref{P-like_for_spaces}.
\end{proof}


\section{Main results}\label{sec:mainresults}

In this section, we 
identify a set $A\subseteq\Omega$  with its characteristic function $\mathbf 1_A$,
and 
view $2^\Omega$  as  the product space $\{0,1\}^\Omega$ with the discrete topology on $\{0,1\}$. Then $2^\Omega$ is a Polish space. The set  $[\Omega]^\omega$ is a $G_\delta$ subset of $ 2^\Omega$, hence it is a  Polish space.
In the same manner, we think of  
$2^\Psi$ and $[\Psi]^\omega$.

\begin{prop}
\label{continous_implies_analytic}
Let $X$ be a Polish space and  
$\rho:\mathcal F\to[\Psi]^\omega$
be a partition regular function.
If $\F$ is an analytic subset of $[\Omega]^\omega$ and $\rho$ is a Borel function, then 
\begin{equation*}
    \mathscr{L}(\rho)\subseteq \Sigma^1_1.
\end{equation*}
\end{prop}

\begin{proof}
Let $d$ be a compatible complete metric on $X$.
Fix a sequence $x=(x_s)_{s\in\Psi}\in X^\Psi$. We need to show that $\Lambda_x(\rho)$ is an analytic subset of $X$. 
Notice that  $\Lambda_{x}(\rho)$ is the projection onto the first coordinate of the set
\begin{equation*}
    A = \left\{(p,F)\in X\times \cF: \forall_{n\in \omega}\, \exists_{K\in [\Omega]^{<\omega}}\, \forall_{s\in \Psi}\, (s\in \rho(F\setminus K)\implies d(p,x_s)<2^{-n})\right\}.
\end{equation*}
Since $A$ is a subset of the Polish space $X\times [\Omega]^\omega$,  the proof will be finished once we show that $A$ is an analytic subset of $X\times [\Omega]^\omega$.
Observe that
\begin{equation*}
    A = \bigcap_{n\in \omega}\bigcup_{K\in [\Omega]^{<\omega}}\bigcap_{s\in \Psi} (B_{n,K,s}\cup C_{n,K,s}),
\end{equation*}
where 
\begin{equation*}
\begin{split}
    B_{n,K,s} &= \{(p,F)\in X\times\cF:s\notin \rho(F\setminus K)\} \,\,\,\text{ and }
    \\
    C_{n,K,s}&=\{(p,F)\in X\times\cF: d(p,x_s)<2^{-n}\}.
\end{split}
\end{equation*}
It is enough to show that the sets $B_{n,K,s}$ and $C_{n,K,s}$ are analytic. 

First, we show that $B_{n,K,s}$ is analytic.
Since 
$B_{n,K,s}=X\times D_{K,s}$, 
where 
$D_{K,s} = \{F\in \cF:s\notin \rho(F\setminus K)\}$, 
the proof will be finished once we show that 
$D_{K,s}$ is analytic.
Notice that 
\begin{equation*}
    \begin{split}
D_{K,s}
&=
\cF\cap \phi^{-1}_K\left[\rho^{-1}\left[\left[\Psi\setminus\{s\}\right]^\omega\right]\right],
    \end{split}
\end{equation*}
where $\phi_K:[\Omega]^\omega\to [\Omega]^\omega$ is given by $\phi_K(A)=A\setminus K$.
Since 
$\cF$ is an analytic set, 
$\phi_K$  is a continuous function, $\rho$ is a Borel function,  
and $\left[\Psi\setminus\{s\}\right]^\omega$ is an open subset of $[\Psi]^\omega$, we obtain that $D_{K,s}$ is analytic.


Finally, we show that $C_{n,K,s}$ is analytic.
Notice that 
$C_{n,K,s}=B(x_s,2^{-n})\times \cF$.
Since $B(x_s,2^{-n})$ is open and $\cF$ is analytic, we obtain that $C_{n,K,s}$ is analytic. 
\end{proof}

It is worth noting that the special case of Proposition \ref{continous_implies_analytic} where $\rho=\rho_\I$, for some coanalytic ideal $\I$ on $\omega$, has been already shown in \cite[Proposition 4.1]{MR3883171}, cf. also \cite[Theorem 4.1]{MR4505549} and \cite[Theorem 3.3]{MSP24}.

We are finally ready to prove our main results.
\begin{thm}\label{thm:mainRamsey}
Let $X$ be an uncountable Polish space. Then 
\begin{equation*}
\mathscr{L}(\rr)=\mathscr{L}(\mathcal{R})=\Sigma_1^1. 
\end{equation*}
Moreover, 
for every nonempty analytic set $C\subseteq X$,  there exists a sequence $y\in X^{[\omega]^2}$ such that $C=\Lambda_y(\rr)=\Lambda_y(\mathcal{R})$. 
\end{thm}

\begin{proof}
The inclusion $\mathscr{L}(\rr)\subseteq\Sigma_1^1$ follows from \cite[Proposition 5.2(2)]{FKK23} and Proposition \ref{continous_implies_analytic}. The inclusion $\mathscr{L}(\mathcal{R})\subseteq\Sigma_1^1$ follows from \cite[Proposition~4.1]{MR3883171} (or Proposition \ref{continous_implies_analytic} again
and \cite[Proposition 5.2(5)]{FKK23}). We will show the ``moreover'' part, from which the converse inclusions $\Sigma_1^1\subseteq \mathscr{L}(\rr)$ and $\Sigma_1^1\subseteq \mathscr{L}(\mathcal{R})$ follow (recall that $\emptyset\in \mathscr{L}(\rr)$ and $\emptyset\in \mathscr{L}(\mathcal{R})$ by definition). 

To this aim, fix a nonempty analytic set $C\subseteq X$. 
Thanks to \cite[Theorem~25.7]{MR1321597}, there exists a Souslin scheme $\{C_s: s\in \omega^{<\omega}\}$ of nonempty closed sets such that 
\begin{equation}\label{eq:souslinscheme}
C = \bigcup_{x\in \omega^\omega}\bigcap_{n\in \omega}C_{x\restriction n}.
\end{equation}
In addition, it can be assumed that $C_s\supseteq C_t$ whenever $s\subseteq t$ for all $t,s\in \omega^{<\omega}$ and $\lim_n \mathrm{diam}(C_{x\restriction n})=0$ for all $x \in \omega^\omega$. 
Note that, since $X$ is complete, for each $x\in \omega^\omega$ there exists $p_x\in X$ such that 
$\bigcap_{n} C_{x\restriction n} = \{p_x\}$. 
Hence $C=\{p_x:x\in \omega^\omega\}$.

Fix any bijection $f:\omega\to\omega^{<\omega}$ such that if $s\subseteq t$ then $f^{-1}(s)\leq f^{-1}(t)$. Define 
\begin{equation*}
A_s=\left\{\{i,j\} \in [\omega]^2: i<j\text{ and }s\subseteq f(i)\subseteq f(j)\right\}
\end{equation*}
for each $s\in\omega^{<\omega}$. Observe that $s\subseteq t$ implies $A_t\subseteq A_s$. Moreover, if $n,k\in\omega$ are distinct, then $A_{s^\frown n}\cap A_{s^\frown k}=\emptyset$.

At this point, define the sequence $y=(y_{\{i,j\}})_{\{i,j\}\in[\omega]^2}$ by 
\begin{equation*}
y_{\{i,j\}}=\begin{cases}
    \,p_{0^\infty} &\text{if }\{i,j\}\notin A_\emptyset,\\
    \,p_{f(\max\{i,j\})^\frown 0^\infty} &\text{if }\{i,j\}\in A_\emptyset.
\end{cases}
\end{equation*}

\setcounter{claim}{0}
\begin{claim}\label{claim:inclusionA1}
    $C\subseteq\Lambda_y(\mathcal{R})$.
\end{claim}
\begin{proof}
    Pick $c\in C$. Then $c=p_x$ for some $x\in\omega^\omega$. For each $n\in\omega$, fix $i_n\in\omega$ such that $f(i_n)=x\restriction n$ and define $F=\{i_n:n\in\omega\}$. Note that the properties of $f$ guarantee that $F$ is infinite and that the sequence $(i_n)_{n\in\omega}$ is increasing. 
We claim that for every neighborhood $U$ of $p_x$ there is a finite set $G\subseteq[\omega]^2$ such that $y_{\{i,j\}}\in U$ for all $\{i,j\}\in [F]^2\setminus G$. 

In fact, fix a neighborhood $U$ of $p_x$, and pick a sufficiently large $n \in \omega$ such that $C_{x\restriction n}\subseteq U$. Let $G=\{\{i_m,i_k\}\in[F]^2: \max\{m,k\}\leq n\}$. Then $G$ is finite. Moreover, if $\{i,j\}\in[F]^2\setminus G$, then $\{i,j\}=\{i_m,i_k\}$ for some $m<k$ with $k>n$. Note that $\{i,j\}\in A_\emptyset$. Then 
\begin{equation*}
y_{\{i,j\}}=p_{f(\max\{i,j\})^\frown 0^\infty}=p_{f(i_k)^\frown 0^\infty}=p_{(x\restriction k)^\frown 0^\infty}\in C_{x\restriction k}\subseteq C_{x\restriction n}\subseteq U.
\end{equation*}
Therefore $c=p_x \in \Lambda_y(\mathcal{R})$. 
\end{proof}

\begin{claim}\label{claim:middleA}
For each $s\in\omega^{<\omega}$ and $F\in[\omega]^\omega$ such that $[F]^2\subseteq A_s$, there exist $n\in\omega$ and $K\in[\omega]^{<\omega}$ for which $[F\setminus K]^2\subseteq A_{s^\frown n}$. 
\end{claim}
\begin{proof}
Fix $s\in\omega^{<\omega}$ and $F\in[\omega]^\omega$ such that $[F]^2\subseteq A_s$. Define $K=\{f^{-1}(s)\}$ and pick $m\in F\setminus K$ such that $|f(m)|=\min\{|f(i)|: i\in F\setminus K\}$. Also, let $n\in\omega$ be such that $s^\frown n\subseteq f(m)$ (which is possible since $[F\setminus K]^2\subseteq [F]^2\subseteq A_s$ implies $s\subseteq f(m)$). Now, we are going to show that $[F\setminus K]^2\subseteq A_{s^\frown n}$. 

For, pick $\{i,j\}\in [F\setminus K]^2$ with $i<j$. Since $\{i,j\}\in[F\setminus K]^2\subseteq [F]^2\subseteq A_s$, we get $f(i)\subseteq f(j)$. Observe that $s^\frown n\subseteq f(m)\subseteq f(i)$. Indeed, either $m=i$ (which is trivial) or the choice of $m$ together with $\{m,i\}\in[F\setminus K]^2\subseteq A_s$ give us $f(m)\subseteq f(i)$. Therefore $s^\frown n\subseteq f(m)\subseteq f(i)\subseteq f(j)$, which means that $\{i,j\}\in A_{s^\frown n}$. 
\end{proof}

\begin{claim}\label{claim:inclusionA2}
    $\Lambda_y(\rr)\subseteq C$.
\end{claim}
\begin{proof}
The claim is trivial if $\Lambda_y(\rr)=\emptyset$. Otherwise, fix $z \in \Lambda_y(\rr)$ and pick $F\in[\omega]^\omega$ such that $(y_{\{i,j\}})_{\{i,j\}\in[F]^2}$  
is $\rr$-convergent to $z$. 

First, suppose that $[F\setminus K]^2\setminus A_\emptyset\neq\emptyset$ for all $K\in[\omega]^{<\omega}$. Then $z=p_{0^\infty} \in C$. Indeed, otherwise, we could find a neighborhood $U$ of $z$ such that $p_{0^\infty}\notin U$. Then there is a finite $K\in[\omega]^{<\omega}$ such that $y_{\{i,j\}}\in U$ for all $\{i,j\}\in[F\setminus K]^2$. However, by our assumption, there is some $\{i,j\}\in[F\setminus K]^2\setminus A_\emptyset$, so for that $\{i,j\}$ we have $y_{\{i,j\}}=p_{0^\infty}\notin U$, which gives a contradiction.

Henceforth, assume that $[F\setminus K]^2\subseteq A_\emptyset$ for some $K\in[\omega]^{<\omega}$. Then applying Claim \ref{claim:middleA} we can recursively define $x\in\omega^\omega$ such that for each $n\in\omega$ there is $K_n$ with $[F\setminus K_n]^2\subseteq A_{x\restriction n}$. We will show that $z=p_x\in C$. 
Indeed, suppose for the sake of contradiction that $z\neq p_x$. Then there is $n\in\omega$ such that $z\notin C_{x\restriction n}$. Then $p_x\in C_{x\restriction n}$ and, since $C_{x\restriction n}$ is closed, $X\setminus C_{x\restriction n}$ is an open neighborhood of $z$. Hence, there exists $K\in[\omega]^{<\omega}$ such that $y_{\{i,j\}}\notin C_{x\restriction n}$ for all $\{i,j\}\in[F\setminus K]^2$. However, $F\setminus (K\cup K_n)\in[\omega]^\omega$ and for each $i<j$ with $\{i,j\}\in [F\setminus (K\cup K_n)]^2\subseteq A_{x\restriction n}$ we have $y_{\{i,j\}}=p_{f(j)^\frown 0^\infty}\in C_{f(j)}\subseteq C_{x\restriction n}$. This is a contradiction.
\end{proof} 

Putting together Claim \ref{claim:inclusionA1} and Claim \ref{claim:inclusionA2}, and using  $\Lambda_y(\mathcal{R})\subseteq \Lambda_y(\rr)$ (by Proposition \ref{prop-lambda-ideal-in-lambda-rho}), we conclude that $C\subseteq \Lambda_y(\mathcal{R}) \subseteq \Lambda_y(\rr)\subseteq C$. Therefore they coincide. 
\end{proof}

\begin{thm}\label{thm:mainHindman}
Let $X$ be an uncountable Polish space. Then
\begin{equation*}
\mathscr{L}(\FS)=\mathscr{L}(\mathcal{H})=\Sigma_1^1.
\end{equation*}
Moreover, for every nonempty analytic set $C\subseteq X$,  there exists a sequence $y\in X^{\omega}$ such that $C=\Lambda_y(\FS)=\Lambda_y(\mathcal{H})$. 
\end{thm}

\begin{proof}
The first part proceeds verbatim as in the proof of Theorem \ref{thm:mainRamsey}, replacing $\rr$ and $\mathcal{R}$ with $\FS$ and $\mathcal{H}$, respectively. Fix a nonempty analytic set $C\subseteq X$ and let $\{C_s: s \in \omega^{<\omega}\}$ be a Souslin scheme as in \eqref{eq:souslinscheme}, with the same properties. Similarly, we need to show that there exists $y \in X^\omega$ such that $C=\Lambda_y(\FS)=\Lambda_y(\mathcal{H})$. 

Thanks to \cite[Lemma 2.2]{MR4356195}, it is possible to fix an infinite set $D\subseteq \omega$ which is ``very sparse'', that is
\begin{itemize}
    \item [(a)] for each $a\in\FS(D)$ there is a unique nonempty finite set $\alpha_D(a)\in[D]^{<\omega}$ such that $a=\sum_{i\in \alpha_D(a)}i$;
    \item [(b)] if $G,H\in [D]^{<\omega}$ are nonempty and $G\cap H\neq\emptyset$, then $\sum_{i\in G}i+\sum_{i\in H}i\notin\FS(D)$.
\end{itemize}
Fix a bijection $f:D\to\omega^{<\omega}$ such that if $s\subseteq t$ then $f^{-1}(s)\leq f^{-1}(t)$, and set 
\begin{equation*}
A_s=\bigcup_{m \in \omega}\left\{\sum_{i=0}^m k_i: 
k_0,\ldots,k_m \in D, \, k_0<\cdots<k_m \text{ and }s\subseteq f(k_0)\subseteq \cdots \subseteq f(k_{m})\right\}
\end{equation*}
for each $s\in\omega^{<\omega}$. 
Observe that $A_s\subseteq \FS(D)$, and that  $s\subseteq t$ implies $A_t\subseteq A_s$. Moreover, since $D$ is very sparse, we have $A_{s^\frown n}\cap A_{s^\frown k}=\emptyset$ for all distinct $n,k \in \omega$.

At this point, define a sequence $y=(y_{i})_{i\in\omega}$ by 
\begin{equation*}
y_{i}=\begin{cases}
    \,p_{0^\infty} &\text{if }i\notin A_\emptyset,\\
    \,p_{f(\max \alpha_D(i))^\frown 0^\infty} &\text{if }i\in A_\emptyset.
\end{cases}
\end{equation*}

\setcounter{claim}{0}
\begin{claim}\label{claim:inclusionB1}
    $C\subseteq\Lambda_y(\mathcal{H})$.
\end{claim}

\begin{proof}
    Pick $c\in C$. Then $c=p_x$ for some $x\in\omega^\omega$. For each $n\in\omega$,  fix $i_n\in D$ such that  $f(i_n)=x\restriction n$ and define $F=\{i_n:n\in\omega\}$. Note that the properties of $f$ guarantee that $F$ is infinite and $(i_n)_{n\in\omega}$ is increasing. 
    We claim that for every neighborhood $U$ of $p_x$ there is $G\in  
     [\omega]^{<\omega}$ such that $y_{i}\in U$ for all $i\in \FS(F)\setminus G$. 

     In fact, fix a neighborhood $U$ of $p_x$, and pick a sufficiently large $n \in \omega$ such that $C_{x\restriction n}\subseteq U$. 
     Let $G=\{i\in\FS(D): \max\alpha_D(i)\leq i_n\}$. Then $G$ is finite. Moreover, if $i\in \FS(F)\setminus G\subseteq\FS(D)\setminus G$, then $i_k=\max\alpha_D(i)>i_n$ for some $k\in\omega$, $k>n$ as $(i_j)_{j\in\omega}$ is increasing and $D$ is very sparse. Note that $i\in A_\emptyset$. Hence 
     $$
     y_{i}=p_{f(\max\alpha_D(i))^\frown 0^\infty}=p_{f(i_k)^\frown 0^\infty}=p_{(x\restriction k)^\frown 0^\infty}\in C_{x\restriction k}\subseteq C_{x\restriction n}\subseteq U.
     $$
Therefore $c=p_x \in \Lambda_y(\mathcal{H})$. 
\end{proof}

\begin{claim}\label{claim:middleB}
    For each $s\in\omega^{<\omega}$ and $F\in[\omega]^\omega$ such that $\FS(F)\subseteq A_s$, 
    there exist $n\in\omega$ and $K\in[\omega]^{<\omega}$ for which $\FS(F\setminus K)\subseteq A_{s^\frown n}$. 
\end{claim}
\begin{proof}
    Fix $s\in\omega^{<\omega}$ and $F\in[\omega]^\omega$ such that $\FS(F)\subseteq A_s$. Observe that $F\subseteq\FS(F)\subseteq A_s\subseteq\FS(D)$. Moreover, if $a,b\in F$ are distinct, then $\alpha_D(a)\cap\alpha_D(b)=\emptyset$. Indeed, otherwise,  $a+b\in\FS(F)\subseteq A_s\subseteq\FS(D)$, which would contradict the fact that $D$ is very sparse. 

Hence, by the above observation, $K=\{a\in F:f^{-1}(s)\in\alpha_D(a)\}$ has at most one element. In particular, $K$ is finite. Pick $e\in \bigcup_{a\in F\setminus K}\alpha_D(a)$ such that $|f(e)|=\min\{|f(i)|: i\in\bigcup_{a\in F\setminus K}\alpha_D(a)\}$. Let $c\in F\setminus K$ be such that $e\in\alpha_D(c)$ (note that such $c$ is unique) and $n\in\omega$ be such that $s^\frown n\subseteq f(e)$ (such $n$ exists by the definition of $K$ and an observation that $e\in\alpha_D(c)$ implies $s\subseteq f(e)$, as $c\in \FS(F)\subseteq A_s$ and $D$ is very sparse). Thus, we are left to show that $\FS(F\setminus K)\subseteq A_{s^\frown n}$.

Fix distinct $a_0,a_1,\ldots,a_m\in F\setminus (K\cup\{c\})$. To complete the proof, we need to prove that $c\in  A_{s^\frown n}$, $\sum_{i\leq m}a_i\in A_{s^\frown n}$, and $c+\sum_{i\leq m}a_i\in A_{s^\frown n}$. Let $(d_i)_{i\leq k}\in D^{k+1}$ be an increasing enumeration of the set $\alpha_D(c)\cup\bigcup_{i\leq m}\alpha_D(a_i)$. Since $\alpha_D(a)\cap\alpha_D(b)=\emptyset$ for all distinct $a,b\in F$, we have 
\begin{equation*}
\sum_{i\leq k} d_i=c+\sum_{i\leq m}a_i\in\FS(F\setminus K)\subseteq\FS(F)\subseteq A_s.
\end{equation*}
Since $D$ is very sparse and $d_0<\cdots<d_k$ in $D$, we get $s\subseteq f(d_0)\subseteq f(d_1)\subseteq\cdots\subseteq f(d_k)$. 
In particular, $f(d_0)\subseteq f(e)$ (because $e=d_j$ for some $j$). By the choice of $e$, we have also $|f(e)|\le |f(d_0)|$, hence $|f(e)|=|f(d_0)|$ and thus $f(e)=f(d_0)$. Since $f$ is injective, we obtain $e=d_0$.  
Hence also $s^\frown n\subseteq f(e)=f(d_0)\subseteq f(d_1)\subseteq\cdots\subseteq f(d_k)$. This shows that $c+\sum_{i\leq m}a_i\in A_{s^\frown n}$. However, since the increasing enumerations of $\alpha_D(c)$ and of $\bigcup_{i\leq m}\alpha_D(a_i)$ are some subsequences of $(d_i)_{i\leq k}$, we can also conclude that $c\in  A_{s^\frown n}$ and $\sum_{i\leq m}a_i\in A_{s^\frown n}$. 
\end{proof}

\begin{claim}\label{claim:inclusionB2}
    $\Lambda_y(\FS)\subseteq C$.
\end{claim}
\begin{proof}
The claim is trivial if $\Lambda_y(\FS)=\emptyset$. Otherwise, fix $z \in \Lambda_y(\FS)$ and pick $F\in[\omega]^\omega$ such that $(y_{i})_{i\in\FS(F)}$  
is $\FS$-convergent to $z$. 

First, suppose that $\FS(F\setminus K)\setminus A_\emptyset\neq\emptyset$ for all $K\in[\omega]^{<\omega}$. Then $z=p_{0^\infty}$. Indeed, in the opposite, we could find a neighborhood $U$ of $z$ such that $p_{0^\infty}\notin U$. Then there is a finite $K\in[\omega]^{<\omega}$ such that $y_{i}\in U$ for all $i\in\FS(F\setminus K)$. However, by our assumption, there is some $i\in\FS(F\setminus K)\setminus A_\emptyset$, so for that $i$ we have $y_{i}=p_{0^\infty}\notin U$, which gives a contradiction.

Hence, assume hereafter that $\FS(F\setminus K)\subseteq A_\emptyset$ for some $K\in[\omega]^{<\omega}$. Then applying Claim \ref{claim:middleB} we can recursively define $x\in\omega^\omega$ such that for each $n\in\omega$ there is $K_n$ with $\FS(F\setminus K_n)\subseteq A_{x\restriction n}$. We will show that $z=p_x\in C$. 
Indeed, suppose for the sake of contradiction that $z\neq p_x$. Then there is $n\in\omega$ such that $z\notin C_{x\restriction n}$. Then $p_x\in C_{x\restriction n}$ and, since $C_{x\restriction n}$ is closed, $X\setminus C_{x\restriction n}$ is an open neighborhood of $z$. Hence, there exists $K\in[\omega]^{<\omega}$ such that $y_{i}\notin C_{x\restriction n}$ for all $i\in\FS(F\setminus K)$. However, $F\setminus (K\cup K_n)\in[\omega]^\omega$ and each $i\in \FS(F\setminus (K\cup K_n))\subseteq A_{x\restriction n}$ is of the form $\sum_{j=0}^m k_j$ for some increasing finite sequence $(k_j)_{j\leq m}\in D^{m+1}$ such that $x\restriction n\subseteq f(k_0)\subseteq f(k_1)\subseteq\cdots\subseteq f(k_m)$. Thus, $y_i=p_{f(k_m)^\frown 0^\infty}\in C_{f(k_m)}\subseteq C_{x\restriction n}$. This is the claimed contradiction. 
\end{proof}

The conclusion follows from putting together Claim \ref{claim:inclusionB1}, Claim \ref{claim:inclusionB2}, and the inclusion  $\Lambda_y(\mathcal{H})\subseteq \Lambda_y(\FS)$ (by Proposition \ref{prop-lambda-ideal-in-lambda-rho}), as in the proof of Theorem \ref{thm:mainRamsey}. 
\end{proof}


\bibliographystyle{amsplain}
\bibliography{rho-references}

@article {MSP24,
    AUTHOR = {Balcerzak, M. and G\l{}\c{a}b, S. and Leonetti, P.},
     TITLE = {Topological complexity of ideal limit points},
PAGES = {1-23}, 
   JOURNAL = {\url{https://arxiv.org/abs/2407.12160}},
year={2024},
}

@article {FKL24,
    AUTHOR = {Filip\'{o}w, R. and Kwela, A. and Leonetti, P.},
     TITLE = {Borel complexity of sets of ideal limit points},
     PAGES = {1-44},
   JOURNAL = {\url{https://arxiv.org/abs/2411.10866}},
year={2025},
}

@article {Gosia,
    AUTHOR = {Filip\'{o}w, R. and Kowalczuk, M. and Kwela, A.},
     TITLE = {Critical ideals for countable compact spaces},
     PAGES = {1-19},
   JOURNAL = {\url{https://arxiv.org/abs/2503.12571}},
year={2025},
}

@article {MR4741322,
    AUTHOR = {Bergelson, V. and Zelada, R.},
     TITLE = {Strongly mixing systems are almost strongly mixing of all
              orders},
   JOURNAL = {Ergodic Theory Dynam. Systems},
  FJOURNAL = {Ergodic Theory and Dynamical Systems},
    VOLUME = {44},
      YEAR = {2024},
    NUMBER = {6},
     PAGES = {1489--1530},
      ISSN = {0143-3857,1469-4417},
   MRCLASS = {37A25 (05D10 37A15)},
  MRNUMBER = {4741322},
       DOI = {10.1017/etds.2023.63},
       URL = {https://doi.org/10.1017/etds.2023.63},
}

@article {MR3019575,
    AUTHOR = {Hru\v{s}\'{a}k, M. and Meza-Alc\'{a}ntara, D.},
     TITLE = {Kat\v{e}tov order, {F}ubini property and {H}ausdorff ultrafilters},
   JOURNAL = {Rend. Istit. Mat. Univ. Trieste},
  FJOURNAL = {Rendiconti dell'Istituto di Matematica dell'Universit\`a di
              Trieste. An International Journal of Mathematics},
    VOLUME = {44},
      YEAR = {2012},
     PAGES = {503--511},
      ISSN = {0049-4704},
   MRCLASS = {03E15 (03C20 03H15)},
  MRNUMBER = {3019575},
MRREVIEWER = {M\'{a}rton Elekes},
}

@phdthesis{alcantara-phd-thesis,
    AUTHOR = {Meza-Alc\'{a}ntara, D.},
     TITLE = {Ideals and filters on countable sets},
    SCHOOL = {Universidad Nacional Aut\'{o}noma de M\'{e}xico},
      YEAR = {2009},
note={(\url{https://ru.dgb.unam.mx/handle/DGB_UNAM/TES01000645364})}, 
URL = {https://ru.dgb.unam.mx/handle/DGB_UNAM/TES01000645364},
}

@book {MR0603625,
    AUTHOR = {Furstenberg, H.},
     TITLE = {Recurrence in ergodic theory and combinatorial number theory},
      NOTE = {M. B. Porter Lectures},
 PUBLISHER = {Princeton University Press, Princeton, NJ},
      YEAR = {1981},
     PAGES = {xi+203},
      ISBN = {0-691-08269-3},
   MRCLASS = {28D05 (10K10 10L10 54H20)},
  MRNUMBER = {603625},
MRREVIEWER = {Michael\ Keane},
}

@article {MR0531271,
    AUTHOR = {Furstenberg, H. and Weiss, B.},
     TITLE = {Topological dynamics and combinatorial number theory},
   JOURNAL = {J. Analyse Math.},
  FJOURNAL = {Journal d'Analyse Math\'ematique},
    VOLUME = {34},
      YEAR = {1978},
     PAGES = {61--85},
      ISSN = {0021-7670,1565-8538},
   MRCLASS = {05A17 (10A45 54H20)},
  MRNUMBER = {531271},
MRREVIEWER = {N.\ Hindman},
       DOI = {10.1007/BF02790008},
       URL = {https://doi.org/10.1007/BF02790008},
}

@article {MR2471564,
    AUTHOR = {Fla{\v s}kov\'a, J.},
     TITLE = {Ideals and sequentially compact spaces},
   JOURNAL = {Topology Proc.},
  FJOURNAL = {Topology Proceedings},
    VOLUME = {33},
      YEAR = {2009},
     PAGES = {107--121},
      ISSN = {0146-4124,2331-1290},
   MRCLASS = {54A20},
  MRNUMBER = {2471564},
MRREVIEWER = {Oleg\ G.\ Okunev},
}

@article {MR0349574,
    AUTHOR = {Hindman, N.},
     TITLE = {Finite sums from sequences within cells of a partition of
              {$N$}},
   JOURNAL = {J. Combinatorial Theory Ser. A},
  FJOURNAL = {Journal of Combinatorial Theory. Series A},
    VOLUME = {17},
      YEAR = {1974},
     PAGES = {1--11},
      ISSN = {0097-3165},
   MRCLASS = {10A45 (04A20)},
  MRNUMBER = {349574},
       DOI = {10.1016/0097-3165(74)90023-5},
       URL = {https://doi.org/10.1016/0097-3165(74)90023-5},
}

@article {MR2948679,
    AUTHOR = {Boja{\'n}czyk, M. and Kopczy{\'n}ski, E. and Toru{\'n}czyk,
              S.},
     TITLE = {Ramsey's theorem for colors from a metric space},
   JOURNAL = {Semigroup Forum},
  FJOURNAL = {Semigroup Forum},
    VOLUME = {85},
      YEAR = {2012},
    NUMBER = {1},
     PAGES = {182--184},
      ISSN = {0037-1912,1432-2137},
   MRCLASS = {05D10 (54D30)},
  MRNUMBER = {2948679},
MRREVIEWER = {N.\ Hindman},
       DOI = {10.1007/s00233-012-9404-4},
       URL = {https://doi.org/10.1007/s00233-012-9404-4},
}

@article {MR4356195,
    AUTHOR = {Filip\'{o}w, R.  and Kowitz, K. and Kwela, A. and
              Tryba, J.},
     TITLE = {New {H}indman spaces},
   JOURNAL = {Proc. Amer. Math. Soc.},
  FJOURNAL = {Proceedings of the American Mathematical Society},
    VOLUME = {150},
      YEAR = {2022},
    NUMBER = {2},
     PAGES = {891--902},
      ISSN = {0002-9939},
   MRCLASS = {54A20 (03E35 03E50 05A17 05C55 11B75)},
  MRNUMBER = {4356195},
MRREVIEWER = {Paolo Leonetti},
       DOI = {10.1090/proc/15720},
       URL = {https://doi.org/10.1090/proc/15720},
}

@article {MR4393937,
    AUTHOR = {He, Xi and Zhang, H. and Zhang, S.},
     TITLE = {The {B}orel complexity of ideal limit points},
   JOURNAL = {Topology Appl.},
  FJOURNAL = {Topology and its Applications},
    VOLUME = {312},
      YEAR = {2022},
     PAGES = {Paper No. 108061, 12},
      ISSN = {0166-8641,1879-3207},
   MRCLASS = {40A35 (40A05 54A20 54H05)},
  MRNUMBER = {4393937},
MRREVIEWER = {Martin\ Dole\v{z}al},
       DOI = {10.1016/j.topol.2022.108061},
       URL = {https://doi.org/10.1016/j.topol.2022.108061},
}

@article {MR4505549,
    AUTHOR = {He, Xi and Zhang, H. and Zhang, S.},
     TITLE = {More on ideal limit points},
   JOURNAL = {Topology Appl.},
  FJOURNAL = {Topology and its Applications},
    VOLUME = {322},
      YEAR = {2022},
     PAGES = {Paper No. 108324, 9},
      ISSN = {0166-8641,1879-3207},
   MRCLASS = {40A35 (40A05 54A20 54H05)},
  MRNUMBER = {4505549},
       DOI = {10.1016/j.topol.2022.108324},
       URL = {https://doi.org/10.1016/j.topol.2022.108324},
}

@article {FKK23,
    AUTHOR = {Filip\'{o}w, R. and Kowitz, K. and Kwela, A.},
     TITLE = {A unified approach to {H}indman, {R}amsey and van der {W}aerden spaces},
     PAGES = {1-53},
year={2024},
   JOURNAL = {to appear in J. Symb. Log. (\url{http://doi.org/10.1017/jsl.2024.8})},
FJOURNAL = {The Journal of Symbolic Logic},
}

@article {MR3883171,
    AUTHOR = {Balcerzak, M. and Leonetti, P.},
     TITLE = {On the relationship between ideal cluster points and ideal
              limit points},
   JOURNAL = {Topology Appl.},
  FJOURNAL = {Topology and its Applications},
    VOLUME = {252},
      YEAR = {2019},
     PAGES = {178--190},
      ISSN = {0166-8641},
   MRCLASS = {40A35 (11B05 40A05 54A20)},
  MRNUMBER = {3883171},
       DOI = {10.1016/j.topol.2018.11.022},
       URL = {https://doi.org/10.1016/j.topol.2018.11.022},
}

@article {MR1181163,
    AUTHOR = {Fridy, J.A.},
    TITLE = {Statistical limit points},
   JOURNAL = {Proc. Amer. Math. Soc.},
  FJOURNAL = {Proceedings of the American Mathematical Society},
    VOLUME = {118},
      YEAR = {1993},
    NUMBER = {4},
     PAGES = {1187--1192},
      ISSN = {0002-9939},
   MRCLASS = {40C99},
  MRNUMBER = {1181163},
MRREVIEWER = {A. Peyerimhoff},
       DOI = {10.2307/2160076},
       URL = {http://dx.doi.org/10.2307/2160076},
}

@book {MR1321597,
    AUTHOR = {Kechris, A. S.},
     TITLE = {Classical descriptive set theory},
    SERIES = {Graduate Texts in Mathematics},
    VOLUME = {156},
 PUBLISHER = {Springer-Verlag, New York},
      YEAR = {1995},
     PAGES = {xviii+402},
      ISBN = {0-387-94374-9},
   MRCLASS = {03E15 (03-01 03-02 04A15 28A05 54H05 90D44)},
  MRNUMBER = {1321597},
MRREVIEWER = {Jakub Jasi\'nski},
       DOI = {10.1007/978-1-4612-4190-4},
       URL = {http://dx.doi.org/10.1007/978-1-4612-4190-4},
}

@article {MR1844385,
    AUTHOR = {Kostyrko, P. and \v{S}al\'{a}t, T. and Wilczy\'{n}ski, W.},
     TITLE = {{$\mathcal{I}$}-convergence},
   JOURNAL = {Real Anal. Exchange},
  FJOURNAL = {Real Analysis Exchange},
    VOLUME = {26},
      YEAR = {2000/01},
    NUMBER = {2},
     PAGES = {669--685},
      ISSN = {0147-1937},
   MRCLASS = {54A20},
  MRNUMBER = {1844385},
MRREVIEWER = {K. Chandrasekhara Rao},
}

@article {MR3883170,
    AUTHOR = {Leonetti, P.},
     TITLE = {Invariance of ideal limit points},
   JOURNAL = {Topology Appl.},
  FJOURNAL = {Topology and its Applications},
    VOLUME = {252},
      YEAR = {2019},
     PAGES = {169--177},
      ISSN = {0166-8641},
   MRCLASS = {40A35 (11B05 40A05 54A20)},
  MRNUMBER = {3883170},
MRREVIEWER = {Jacek Tryba},
       DOI = {10.1016/j.topol.2018.11.016},
       URL = {https://doi.org/10.1016/j.topol.2018.11.016},
}

@article {MR3920799,
    AUTHOR = {Leonetti, P. and Maccheroni, F.},
     TITLE = {Characterizations of ideal cluster points},
   JOURNAL = {Analysis (Berlin)},
  FJOURNAL = {Analysis. International Mathematical Journal of Analysis and
              its Applications},
    VOLUME = {39},
      YEAR = {2019},
    NUMBER = {1},
     PAGES = {19--26},
      ISSN = {0174-4747},
   MRCLASS = {40A35 (54A20)},
  MRNUMBER = {3920799},
       DOI = {10.1515/anly-2019-0001},
       URL = {https://doi.org/10.1515/anly-2019-0001},
}

\end{document}